\documentclass[11pt]{article}

\usepackage{enumerate, amsmath, amsthm, amsfonts, amssymb, xy,  mathrsfs, graphicx, paralist, lscape}
\usepackage[usenames, dvipsnames]{color}
\usepackage[margin=1in]{geometry} 
\usepackage[colorlinks=true, urlcolor=blue, linkcolor=blue, citecolor=blue]{hyperref}

\input xy
\xyoption{all}

\numberwithin{equation}{section}
\newtheorem{theorem}[equation]{Theorem}

\newtheorem{proposition}[equation]{Proposition}
\newtheorem{lemma}[equation]{Lemma}

\newtheorem{conjecture}[equation]{Conjecture}

\theoremstyle{definition}
\newtheorem{rmk}[equation]{Remark}
\newenvironment{remark}[1][]{\begin{rmk}[#1] \pushQED{\qed}}{\popQED \end{rmk}}
\newtheorem{eg}[equation]{Example}
\newenvironment{example}[1][]{\begin{eg}[#1] \pushQED{\qed}}{\popQED \end{eg}}
\newtheorem{defn}[equation]{Definition}
\newenvironment{definition}[1][]{\begin{defn}[#1]\pushQED{\qed}}{\popQED \end{defn}}

\newcommand{\bF}{\mathbf{F}}

\newcommand{\rH}{\mathrm{H}}

\newcommand{\bK}{\mathbf{K}}

\newcommand{\bL}{\mathbf{L}}

\newcommand{\bO}{\mathbf{O}}

\newcommand{\bS}{\mathbf{S}}

\newcommand{\rS}{\mathrm{S}}

\newcommand{\bZ}{\mathbf{Z}}

\newcommand{\fg}{\mathfrak{g}}

\renewcommand{\phi}{\varphi}
\renewcommand{\emptyset}{\varnothing}

\renewcommand{\tilde}[1]{\widetilde{#1}}

\newcommand{\ul}[1]{\underline{#1}}
\newcommand{\DS}{\displaystyle}

\newcommand{\arxiv}[1]{\href{http://arxiv.org/abs/#1}{{\tt arXiv:#1}}}

\makeatletter
\def\Ddots{\mathinner{\mkern1mu\raise\p@
\vbox{\kern7\p@\hbox{.}}\mkern2mu
\raise4\p@\hbox{.}\mkern2mu\raise7\p@\hbox{.}\mkern1mu}}
\makeatother

\DeclareMathOperator{\im}{image}

\renewcommand{\hom}{\operatorname{Hom}}

\newcommand{\GL}{\mathbf{GL}}

\newcommand{\SO}{\mathbf{SO}}

\newcommand{\fgl}{\mathfrak{gl}}

\newcommand{\fpe}{\mathfrak{pe}}

\newcommand{\JPW}{\mathbf{JPW}}
\newcommand{\uJPW}{\ul{\rm JPW}}
\newcommand{\La}{\mathbf{La}}
\newcommand{\uLa}{\ul{\rm La}}

\usepackage{tikz}

\title{Derived supersymmetries of determinantal varieties}
\author{Steven V Sam}
\date{July 13, 2012}

\begin{document}

\maketitle

\begin{abstract}
We show that the linear strands of the Tor of determinantal varieties in spaces of symmetric and skew-symmetric matrices are irreducible representations for the periplectic (strange) Lie superalgebra. The structure of these linear strands is explicitly known, so this gives an explicit realization of some representations of the periplectic Lie superalgebra. This complements results of Pragacz and Weyman, who showed an analogous statement for the generic determinantal varieties and the general linear Lie superalgebra. We also give a simpler proof of their result. Via Koszul duality, this is an odd analogue of the fact that the coordinate rings of these determinantal varieties are irreducible representations for a certain classical Lie algebra.
\end{abstract}

\tableofcontents

\section*{Introduction.}

In this paper, we are interested in the symmetries of three classes of varieties known as determinantal varieties. Let $E, F$ be vector spaces. We consider the space of generic matrices $\hom(E,F)$, symmetric matrices $S^2 E$, and skew-symmetric matrices $\bigwedge^2 E$. These carry natural actions of general linear groups $\GL(E) \times \GL(F)$, $\GL(E)$, and $\GL(E)$, respectively, via change of basis of the vector spaces. The orbits under these group actions are classified by the rank of the matrix, and the orbit closures are the determinantal varieties. The algebraic and geometric properties of these varieties have been intensely studied in the past (see \cite{brunsvetter} for a general reference), and the group action is incredibly useful as a computational and theoretical tool since one gets induced actions on all ``functorial'' constructions, such as the coordinate ring, the minimal free resolution, the local cohomology, etc.

These induced actions are an ``obvious symmetry'' and the point of departure of this paper is that there are additional ``hidden symmetries''. Before explaining our results, we highlight some of the previous literature. First, the group actions can be replaced by infinitesimal actions of their Lie algebras $\fgl(E)\times \fgl(F)$ and $\fgl(E)$. For the coordinate ring of a determinantal variety in $\hom(E,F)$, there is an action of $\hom(E,F)^* = \hom(F,E)$, which is multiplication by linear forms, and $\hom(F,E)$ is the lower triangular part of the block decomposition of the larger Lie algebra $\fgl(E \oplus F)$, while $\fgl(E) \times \fgl(F)$ forms the block diagonal part. It was shown by Levasseur and Stafford \cite{levasseur} that (after a suitable character twist of the action of $\fgl(E) \times \fgl(F)$) one can extend the above two actions to an action of the whole Lie algebra $\fgl(E \oplus F)$ by having $\hom(E,F)$ act by certain differential operators on the coordinate ring of a determinantal variety. Furthermore, the coordinate ring becomes a (non-integrable) {\it irreducible} highest weight representation of $\fgl(E \oplus F)$. \cite{levasseur} also handled determinantal varieties in the space of symmetric or skew-symmetric matrices (the large Lie algebra $\fgl(E \oplus F)$ is replaced by either a symplectic or orthogonal Lie algebra, respectively). These varieties are related to the classical Hermitian symmetric pairs. The results were extended in \cite{tanisaki} to all Hermitian symmetric pairs, and explicit formulas for the differential operators and highest weights are given. 

Enright and Willenbring \cite{enrightwillenbring} showed that the action of the large Lie algebra carries over to the entire minimal free resolution of the determinantal varieties (it is an analogue of the Bernstein--Gelfand--Gelfand resolution), and the extension to all Hermitian symmetric pairs appears in \cite{enrighthunziker}. One could think of this as a ``vertical'' hidden symmetry. In fact, there is an additional ``horizontal'' hidden symmetry on the minimal free resolution of a determinantal variety. More specifically, on the linear strands of the resolution, there is an action of $\hom(E,F)$ given by applying the differentials. This time, one interprets $\hom(E,F)$ as the lower-triangular part of the block decomposition of the Lie {\it super}algebra $\fgl(E|F)$. Again, $\fgl(E) \times \fgl(F)$ forms the block diagonal part. It was shown by Pragacz and Weyman \cite{pw} (see also \cite{aw1, aw2, aw3} for further developments) that one can extend the above two actions to an action of the whole Lie superalgebra $\fgl(E|F)$ (again after a suitable character twist). The linear strands tensored with the residue field (i.e., Tor) become irreducible highest weight representations of $\fgl(E|F)$, with the exception of the degenerate case of rank 0 matrices, i.e., the Koszul complex.

We remark that some other interactions between Lie superalgebras and free resolutions (related to degeneracy loci) appear in \cite[\S A.6]{pragacz} and \cite{sam}.

The goal of this paper is to give a simpler proof of the existence of this superalgebra action (Theorem~\ref{thm:lascoux_action}) on the Tor of the determinantal varieties in $\hom(E,F)$ as well as to construct an analogous action (Theorem~\ref{thm:periplecticaction}) for the Tor of the (skew-)symmetric determinantal varieties (it turns out these two cases essentially collapse to one) as well as some other modules supported in the determinantal varieties which are related to classical invariant theory (Remark~\ref{rmk:invttheory}). The simplest non-trivial case is given in Example~\ref{eg:s=1}. The superalgebra $\fgl(E|F)$ is replaced by the periplectic Lie algebra, which is a super analogue of the symplectic Lie algebra. A foreshadowing of this result and some small cases are contained in \cite{jpw}, but as in \cite{pw}, the word ``superalgebra'' never appears.

The main idea of the paper is to use a variant of ``Weyl's construction'' for representations of the classical groups \cite[\S 17.3, \S 19.5]{fultonharris}. The idea is to start with the vector representation of a classical group, apply a Schur functor to it, and then mod out by the image of a map from a smaller Schur functor to obtain an irreducible representation. By semisimplicity, one could instead map to a smaller Schur functor and take the kernel. For the superalgebras of interest in this paper, the main point is that the vector representations, and their duals, can be interpreted as 2-term chain complexes, and we can construct Schur complexes (see \cite{abw} or \cite[\S 2.4]{weyman}) from them. One point of difficulty is that these superalgebras do not have a semisimple representation theory, so we have to combine the two approaches to Weyl's construction. The Schur complexes give super analogues of most of the rich combinatorics and invariant theory that one associates with Schur functors. In particular, they are the irreducible polynomial representations for the superalgebra $\fgl(E|F)$, which were studied in \cite{bereleregev}. In fact, the category of polynomial representations is semisimple. In our construction, we mix the Schur complexes on the vector representation with the Schur complexes on the dual vector representation, so we leave the polynomial category. Also, the vector representation of the periplectic Lie superalgebra is isomorphic to its dual up to grading shift. So the Weyl construction above is a transition from the classical semisimple setting.

We finish the introduction with an outline of the paper. In \S\ref{sec:liesuper}, we give concrete realizations of the two classes of Lie superalgebras that we will discuss. In particular, they are $\bZ$-graded, and in \S\ref{sec:Zgraded}, we introduce some notation to think about $\bZ$-graded representations in terms of chain complexes. In \S\ref{sec:skewsymminors} we state and prove that the linear strands of the Tor of the (skew-)symmetric determinantal varieties are irreducible representations for the periplectic superalgebra. In \S\ref{sec:genericminors} we discuss the case of generic determinantal varieties. The proofs are similar, and the result in this case is already known, so we only mention the differences from \S\ref{sec:skewsymminors}.

\subsection*{Notation.}

\begin{compactitem}[$\bullet$]
\item $\bZ$ denotes the set of all integers

\item We work over a field $K$ of characteristic 0 throughout. Unless otherwise stated, all multilinear operations are taken with respect to $K$. We use $\rS^k$ and $\bigwedge^k$ to denote symmetric and exterior powers, respectively. We also use the notation $\rS^\bullet = \bigoplus_{k \ge 0} \rS^k$ and $\bigwedge^\bullet = \bigoplus_{k \ge 0} \bigwedge^k$. 

\item We use $\det$ to denote the top exterior power of a vector space. The dual of a vector space $E$ is denoted $E^*$. Given two vector spaces $E,F$, we use $\hom(E,F)$ to denote the space of all linear maps $E \to F$. Given a matrix $\phi$, we use $\phi^T$ to denote the transpose matrix. The general linear group is $\GL(E)$, and its Lie algebra is $\fgl(E)$.

\item A partition $\lambda = (\lambda_1, \lambda_2, \dots)$ is a weakly decreasing sequence of nonnegative integers with $|\lambda| := \sum_i \lambda_i < \infty$. We use $\ell(\lambda)$ to denote the largest $r$ such that $\lambda_r > 0$. We visualize partitions by Young diagrams: left-justified collections of boxes such that there are $\lambda_i$ boxes in the row $i$ (counted from top to bottom). The transpose partition $\lambda^\dagger$ is the partition obtained by counting column lengths of the Young diagram of $\lambda$. We will also use exponential notation for partitions, i.e., $s^d$ denotes the sequence $(s, s, \dots, s)$ ($d$ times). We also make use of skew Young diagrams: if the diagram of $\mu$ is a subset of the diagram of $\lambda$, i.e., $\mu_i \le \lambda_i$ for all $i$, then $\lambda / \mu$ is the set complement of these diagrams.

\item Given a partition $\lambda$, $\bS_\lambda$ denotes the Schur functor associated to $\lambda$. This follows the notation $\bK_\lambda$ in \cite[\S 2.1]{weyman}. In particular, if $\ell(\lambda)=1$, then $\bS_\lambda = \rS^{|\lambda|}$ is a symmetric power, and if $\ell(\lambda^\dagger)=1$, then $\bS_\lambda = \bigwedge^{|\lambda|}$ is an exterior power. The functor $\bL_\lambda$ in \cite[\S 2.1]{weyman} is isomorphic to $\bS_{\lambda^\dagger}$ since we work over a field of characteristic 0. We will also use skew Schur functors $\bS_{\lambda / \mu}$ which can also be found in \cite[\S 2.1]{weyman}.

\item Given a graded $K$-algebra $A$, and vector spaces $U,V$, we will use the notation
\[
U \otimes A(-i) \to V \otimes A
\]
to denote a map between free $A$-modules such that in matrix form, the entries consist of homogeneous elements in $A$ of degree $i$. We might also write $U \otimes A(-i-j) \to V \otimes A(-j)$ if we need to compose such maps.

\item All complexes $\bF_\bullet$ are graded in homological notation, i.e., the differential $d \colon \bF_i \to \bF_{i-1}$ lowers the degree of the index. The dual complex $\bF_\bullet^*$ has grading $\bF_i^* = \bF_{-i}$. We denote homological grading shift by $[i]$, i.e., $\bF[i]_j = \bF_{i+j}$.

\item A superspace is a $\bZ/2$-graded vector space $V$, and we will write it as $V_0 \oplus V_1$. The dimension of $V$ is $(\dim V_0 | \dim V_1)$. We will use $V[1]$ to denote the superspace $V_1 \oplus V_0$.

\item Given a complex $\bF_\bullet$ or a superspace $V$, we can define Schur complexes $\bS_\lambda(\bF_\bullet)$ or super analogues of Schur functors $\bS_\lambda V$. The construction is analogous in both cases, and references for this construction are \cite[\S V]{abw} and \cite[\S 2.4]{weyman}. Skew versions will also be used.
\end{compactitem}

\subsection*{Acknowledgements.}

I thank Jerzy Weyman for suggesting this problem. I thank Andrew Snowden and Jerzy Weyman for carefully reading a draft of this article. 

The author was supported by an NDSEG fellowship while this work was done.

\section{Lie superalgebras.} \label{sec:liesuper}

We will not give the general definitions of Lie superalgebras. Rather, in this section, we just give concrete realizations of the superalgebras that will appear in this paper. For general background, we refer the reader to \cite{kac, kacsketch}.

\subsection{General linear Lie superalgebras.}

Fix positive integers $n,m$ and consider the space of $(n+m) \times (n+m)$ block matrices 
\[
\fgl(n|m) = \left\{ \begin{pmatrix} A & B \\ C & D \end{pmatrix} \right\},
\]
where $A$ is $n \times n$, $B$ is $n \times m$, $C$ is $m \times n$ and $D$ is $m \times m$.

We put a $\bZ$-grading on $\fgl(n|m)$ by
\begin{align*}
\fgl(n|m)_{-1} &= \left\{ \begin{pmatrix} 0 & 0 \\ C & 0 \end{pmatrix} \right\}, \qquad  
\fgl(n|m)_0 = \left\{ \begin{pmatrix} A & 0 \\ 0 & D \end{pmatrix} \right\}, \qquad
\fgl(n|m)_{1} = \left\{ \begin{pmatrix} 0 & B \\ 0 & 0 \end{pmatrix} \right\}
\end{align*}
and for homogeneous elements $X,Y \in \fgl(n|m)$ of homogeneous degrees $\deg(X), \deg(Y)$, we define a bracket 
\[
[X,Y] = XY - (-1)^{\deg(X) \deg(Y)} YX.
\]
Then $\fgl(n|m)$ is a Lie superalgebra via the bracket $[,]$.

We can also define $\fgl(n|m)$ in a basis-free way. Let $V = E \oplus F$ be a $\bZ/2$-graded vector space of dimension $(n|m)$. Then $\fgl(n|m)$ is identified with the space of endomorphisms of $V$ with the natural $\bZ$-grading, and we can write
\[
\fgl(V)_{-1} = \hom(F,E), \qquad \fgl(V)_0 = \fgl(E) \times \fgl(F), \qquad \fgl(V)_1 = \hom(E,F).
\]
The Lie bracket is $\GL(E) \times \GL(F)$-equivariant.

\subsection{Periplectic superalgebras.} \label{sec:periplectic}

We define the periplectic Lie superalgebra as the subalgebra of $\fgl(n|n)$ of matrices of the form
\begin{align*}
\fpe(n) = \left\{ \begin{pmatrix} A & B \\ C & -A^T \end{pmatrix} \ \middle|\  B=-B^T,\ C=C^T\right\}.
\end{align*}
It is straightforward to check that $\fpe(n)$ is closed under the Lie bracket $[,]$ and that $\fpe(n)$ inherits the $\bZ$-grading from $\fgl(n|n)$.

We can also define $\fpe(n)$ as the subalgebra of $\fgl(n|n)$ which preserves the odd skew-symmetric bilinear form represented by the matrix $\omega = \begin{pmatrix} 0 & I_n \\ -I_n & 0 \end{pmatrix}$. Let us pause to say what exactly this means since one has to be careful with signs. We define the supertranspose on $\fgl(n|m)$ by 
\[
\begin{pmatrix} A & B \\ C & D \end{pmatrix}^{ST} = \begin{pmatrix} A^T & -C^T \\ B^T & D^T \end{pmatrix}.
\]
Then a homogeneous element $X \in \fgl(n|n)$ preserves $\omega$ if 
\[
X^{ST}\omega + (-1)^{\deg(X)} \omega X = 0.
\]
This definition is an odd analogue of the definition of the symplectic Lie algebra.

We can also define $\fpe(n)$ in a basis-free way. First write $\fgl(n|n)$ as $\fgl(E \oplus F)$ from the previous section. The bilinear form $\omega$ defines an isomorphism $E \cong F^*$, and the $\bZ$-grading of $\fpe(V)$ becomes
\begin{align*}
\fpe(V)_{-1} = \rS^2 E, \qquad \fpe(V)_{0} = \fgl(E), \qquad \fpe(V)_{1} = \bigwedge^2 E^*.
\end{align*}
The Lie bracket is $\GL(E)$-equivariant. 

\section{$\bZ$-graded representations and 2-sided complexes.} \label{sec:Zgraded}

Although everything in the language of superalgebras is only $\bZ/2$-graded, we have seen that for the two algebras of interest in this paper, $\fgl(V)$ and $\fpe(V)$, the $\bZ/2$-grading can be lifted to a $\bZ$-grading. Here we will develop some notation for thinking about those representations which carry a compatible $\bZ$-grading.

We will reinterpret the $\bZ$-grading on a representation as a certain pair of complexes. Before beginning, let us elaborate on why using complexes is  relevant for studying $\bZ$-graded representations. For both of our Lie superalgebras $\fg$, the bracket of any two elements in $\fg_1$ is 0 (similarly for two elements in $\fg_{-1}$). Explicitly, this means that they anticommute with one another in any representation. On the level of the universal enveloping algebra, this gives one an action of an exterior algebra, and this is the object which acts on complexes (via the differential).

The idea behind this section comes from \cite{jpw}.

\subsection{General linear case.} \label{sec:gl_2sided}

The vector representation of $\fgl(V)$ is $V = E \oplus F$ with the action of matrix multiplication. We can view $V$ as a 2-term complex in the following way. First, we have maps
\begin{align*}
F \otimes \hom(F,E) &\to E\\
E \otimes \hom(E,F) &\to F
\end{align*}
given by evaluation, and this coincides with the action of $\fgl(V)_{-1}$ and $\fgl(V)_1$ on $V$.

Now let $A = \rS^\bullet(\hom(F,E)^*)$ be the (graded) coordinate ring of the vector space $\hom(F,E)$. By adjunction, we can rewrite the evaluation map $F \otimes \hom(F,E) \to E$ as $F \to E \otimes \hom(F,E)^*$. We can extend this map $A$-linearly to get 
\[
\Phi \colon F \otimes A(-1) \to E \otimes A,
\]
If we choose bases for $E$ and $F$, then $\Phi$ can be represented by a matrix of linear forms on $\hom(F,E)$, and to recover the action of $X \in \hom(F,E)$, we simply evaluate $\Phi(X)$. Of course, along with the map $\Phi$, we also have the map
\[
\Phi' \colon E \otimes A'(-1) \to F \otimes A'
\]
where $A' = \rS^\bullet(\hom(E,F)^*)$, and the maps $\Phi$ and $\Phi'$ satisfy certain compatibility relations coming from the fact that they come from the action of the Lie superalgebra $\fgl(V)$. We encode this structure in the next definition.

\begin{definition}
A {\bf 2-sided $\fgl(V)$-complex} is a sequence of $\fgl(E) \times \fgl(F)$-modules $(\bF_i)_i$ with equivariant maps $\Phi_i \colon \bF_i \otimes A(-1) \to \bF_{i-1} \otimes A$ and $\Phi'_{i-1} \colon \bF_{i-1} \otimes A'(-1) \to \bF_i \otimes A'$ such that for all $X \in \hom(F,E)$ and $Y \in \hom(E,F)$, 
\begin{compactenum}
\item $\Phi_{i+1}(X) \Phi'_i(Y) - \Phi'_{i-1}(Y) \Phi_i(X) \colon \bF_i \to \bF_i$ coincides with the action of $[X,Y] \in \fgl(E) \times \fgl(F)$, 
\item $\Phi_{i+1}(X) \Phi_i(X) = 0$, and
\item $\Phi'_i(Y) \Phi'_{i-1}(Y) = 0$. \qedhere
\end{compactenum}
\end{definition}

Condition 2 implies that for all $X, X' \in \hom(F,E)$, we have $\Phi_{i+1}(X) \Phi_i(X') = -\Phi_{i+1}(X') \Phi_i(X)$ (apply it to $X+X'$). Similar remarks for condition 3 hold also.

We put $\bF_0 = E$ and $\bF_1 = F$ for the vector representation.

There is an obvious notion of morphisms between 2-sided $\fgl(V)$-complexes and the tensor product of two 2-sided $\fgl(V)$-complexes. The following is immediate.

\begin{proposition}
The tensor category of $2$-sided $\fgl(V)$-complexes is equivalent to the tensor category of $\bZ$-graded representations of $\fgl(V)$.
\end{proposition}

The advantage of this viewpoint is that we will be able to compare certain tensor constructions on $\Phi$ with the linear strands of certain free resolutions over the polynomial ring $A$. 

\subsection{Periplectic case.} \label{sec:peri_2sided}

The same kinds of definitions can be made for $\fpe(V)$. The vector representation of $\fpe(V)$ is $V = E \oplus E^*$, again with the action of matrix multiplication. As before, we have evaluation maps
\begin{align*}
E^* \otimes \rS^2 E &\to E,\\
E \otimes \bigwedge^2 E^* &\to E^*,
\end{align*}
and this coincides with the action of $\fpe(V)_{-1}$ and $\fpe(V)_1$ on $E \oplus E^*$. Let $B = \rS^\bullet(\rS^2 E^*)$ and $B' = \rS^\bullet(\bigwedge^2 E)$. As before, we get maps
\begin{align*}
\Phi \colon E^* \otimes B(-1) &\to E \otimes B,\\
\Phi' \colon E \otimes B'(-1) &\to E^* \otimes B'.
\end{align*}

\begin{definition}
A {\bf 2-sided $\fpe(V)$-complex} is a sequence of $\fgl(E)$-modules $(\bF_i)_i$ with equivariant maps $\Phi_i \colon \bF_i \otimes B(-1) \to \bF_{i-1} \otimes B$ and $\Phi'_{i-1} \colon \bF_{i-1} \otimes B'(-1) \to \bF_i \otimes B'$ such that for all $X \in \rS^2 E$ and $Y \in \bigwedge^2 E^*$, 
\begin{compactenum}
\item $\Phi_{i+1}(X) \Phi'_i(Y) - \Phi'_{i-1}(Y) \Phi_i(X) \colon \bF_i \to \bF_i$ coincides with the action of $[X,Y] \in \fgl(E)$, 
\item $\Phi_{i+1}(X) \Phi_i(X) = 0$, and
\item $\Phi'_i(Y) \Phi'_{i-1}(Y) = 0$. \qedhere
\end{compactenum}
\end{definition}

We put $\bF_0 = E$ and $\bF_1 = E^*$ for the vector representation.

As before, the following is immediate from the definitions.

\begin{proposition}
The tensor category of $2$-sided $\fpe(V)$-complexes is equivalent to the tensor category of $\bZ$-graded representations of $\fpe(V)$.
\end{proposition}

When the context about which algebra we are discussing is clear, we will simply use the phrase ``2-sided complex''.

\section{(Skew-)symmetric minors.} \label{sec:skewsymminors}

We continue to use the notation $B = \rS^\bullet(\rS^2 E^*)$ and $B' = \rS^\bullet(\bigwedge^2 E)$ from \S\ref{sec:peri_2sided}.

\subsection{JPW complexes.} \label{sec:per_jpw}

Choose nonnegative integers $s$ and $r$ so that $\dim E > s+r$. Given a partition $\alpha$ with $\ell(\alpha) \le s$, define the partition
\begin{align} \label{eqn:Prs}
P_{r,s}(\alpha) = (s + \alpha_1, \dots, s + \alpha_s, s^r, \alpha^\dagger_1,
\dots, \alpha^\dagger_{\alpha_1}),
\end{align}
which we visualize as follows:
\[
\begin{tikzpicture}[scale=.4]
\draw (0,-2) -- (0,7) -- (8,7) -- (8,6) -- (6,6) -- (6,5) -- (5,5) -- (5,4) -- (3,4) -- (3,1) -- (2,1) -- (2,0) -- (1,0) -- (1,-2) -- (0,-2) -- cycle;
\draw (0,3) -- (3,3);
\draw (0,4) -- (3,4) -- (3,7);
\path (1.5,5.5) node {$s \times s$};
\path (1.5,3.5) node {$r \times s$};
\path (4.5,5.5) node {$\alpha$};
\path (1.5,1.5) node {$\alpha^\dagger$};
\end{tikzpicture}
\]

Set
\begin{align} \label{eqn:jpwdefn}
\JPW^{r,s}_i = \bigoplus_{|\alpha| = i} \bS_{P_{r,s}(\alpha)} E^*
\end{align}
which naturally carries an action of $\fgl(E)$. Up to a homological grading shift, the sequences $\JPW^{r,s}_\bullet \otimes B$ can be realized as the linear strands of certain minimal free resolutions over the polynomial ring $B$. More specifically, when $s$ is even, they appear in the minimal free resolution of the ideal of $(r+2) \times (r+2)$ minors of the generic symmetric matrix $\rS^2 E$ \cite[Theorem 6.3.1(c)]{weyman}. When $s$ is odd and $r$ is odd, $\JPW^{r,s}_\bullet \otimes B$ is a linear strand in the minimal free resolution of the module $M_2$ mentioned on \cite[p. 180]{weyman}. The case of $s$ odd and $r$ even is not mentioned in \cite{weyman}, but in this case, $\JPW^{r,s}_\bullet \otimes B$ can be realized as the linear strand in the minimal free resolution of the module obtained from $M_2$ via the localization trick in \cite[\S 6.3, part 3]{weyman}. The construction of these complexes first appears in \cite{jpw}.

\begin{remark} \label{rmk:invttheory}
We pause to point out the significance of the module $M_2$ mentioned above. Consider a vector space $W$ equipped with a (split) orthogonal form, and let $\SO(W)$ and $\bO(W)$ be the special orthogonal and general orthogonal groups. The ring of $\bO(W)$-invariant polynomials on $W^{\oplus N}$ is naturally isomorphic to the coordinate ring $R$ of the determinantal variety of symmetric $N \times N$ matrices with rank at most $\dim W$ \cite[Theorem 10.4.0.3]{smt}. The ring of $\SO(W)$-invariant polynomials is a degree 2 extension of $R$, and as an $R$-module, it is $R \oplus M_2$. This direct sum decomposition is the eigenspace decomposition of the action of $\bZ/2 \cong \bO(V) / \SO(V)$.
\end{remark}

As a consequence of the above discussion, we get $B$-linear maps
\[
\Phi_i \colon \JPW^{r,s}_i \otimes B(-1) \to \JPW^{r,s}_{i-1} \otimes B.
\]
The main result in this section is that $\Phi$ can be completed to a 2-sided complex $(\Phi, \Phi')$.

\begin{theorem} \label{thm:periplecticaction}
There exist $B'$-linear maps
\begin{align*}
\Phi'_i \colon \JPW_i^{r,s} \otimes B'(-1) \to \JPW_{i+1}^{r,s} \otimes B'
\end{align*}
so that $(\Phi, \Phi')$ is a $2$-sided $\fpe(V)$-complex. In particular, $\JPW_\bullet^{r,s}$ affords an action of $\fpe(V)$. Moreover, $\JPW_\bullet^{r,s}$ is an irreducible $\fpe(V)$-representation.
\end{theorem}

The proof will be given in \S\ref{sec:peri_exist}.

\begin{example}  \label{eg:s=1}
Before handling the general case, we illustrate the case of $s=1$.

Set $n = \dim E$ and $k = n - 1 - r$. Start with the vector representation $V$, which corresponds to the 2-sided complex 
\begin{align*}
\Phi \colon E^* \otimes B(-1) \to E \otimes B,\\
\Phi' \colon E^* \otimes B' \gets E \otimes B'(-1),
\end{align*}
and consider the representation $W = \bigwedge^k V \otimes \det E^*$, which corresponds to the 2-sided complex
\begin{align*}
\det E^* \otimes \rS^k E^* \otimes B(-k) \to \bigwedge^{n-1} E^* \otimes \rS^{k-1} E^* \otimes B(-k+1) \to \cdots \to \bigwedge^{r+1} E^* \otimes B,\\
\det E^* \otimes \rS^k E^* \otimes B' \gets \bigwedge^{n-1} E^* \otimes \rS^{k-1} E^* \otimes B'(-1) \gets \cdots \gets \bigwedge^{r+1} E^* \otimes B'(-k).
\end{align*}
So $W_i = \bigwedge^{r+1+i} E^* \otimes \rS^i E^*$. From this description, we can find a $\fpe(V)$-subrepresentation of $W$. First note that $W_i$ is an irreducible $\fgl(E)$-module for $i=0$ or $i=k$. Otherwise we have
\[
W_i = \bS_{(i,1^{r+1+i})} E^* \oplus \bS_{(i+1,1^{r+i})} E^*
\]
by Pieri's rule \cite[Corollary 2.3.5]{weyman}. Again by Pieri's rule, we see that $\bS_{(i,1^{r+i+1})} E^*$ is not a direct summand of $\bS_{(i,1^{r+i-1})} E^* \otimes \rS^2 E^*$. Since $\Phi$ and $\Phi'$ are $\fgl(E)$-equivariant, we see that in the map $\Phi \colon W_i \to W_{i-1} \otimes \rS^2 E^*$, the summand $\bS_{(i,1^{r+1+i})} E^*$ can only map to $\bS_{(i-1,1^{r+i})} E^*$. Similarly, in the map $\Phi' \colon W_i \to W_{i+1}$, the summand $\bS_{(i,1^{r+1+i})} E^*$ can only map to $\bS_{(i+1,1^{r+2+i})} E^*$. 

So we get a subrepresentation $W' \subset W$ given by $W'_i = \bS_{(i,1^{r+1+i})} E^*$. We also see that $W / W' = \JPW^{r,1}_\bullet$, so we deduce that $\JPW^{r,1}_\bullet$ can be given the structure of a 2-sided complex.
\end{example}

\begin{remark}
The quotient $W \to \JPW^{r,1}_\bullet \to 0$ from Example~\ref{eg:s=1} can be extended to a long exact sequence
\[
\cdots \to (\bigwedge^{k-4} V)[2] \otimes \det E^* \to (\bigwedge^{k-2} V)[1] \otimes \det E^* \to \bigwedge^k V \otimes \det E^* \to \JPW^{r,1}_\bullet\to 0
\]
where the differentials are induced by the trace map $K[1] \to \bigwedge^2 V$, which is defined in Proposition~\ref{prop:periplectic_trace}.
\end{remark}

\subsection{Trace and evaluation.} \label{sec:periplectic:morphisms}

\begin{proposition} \label{prop:periplectic_trace}
We have nonzero $\fpe(V)$-equivariant maps $K[1] \to \bigwedge^2 V$ and $\rS^2 V \to K[1]$, where $K$ denotes the trivial $1$-dimensional module. We call these maps {\bf trace} and {\bf evaluation}, respectively.
\end{proposition}

These maps also appeared in \cite[\S 4]{jpw}.

\begin{proof}
In terms of 2-sided complexes, we can write $\bigwedge^2 V$ as
\begin{align*}
\rS^2 E^* \otimes B(-2) \to E \otimes E^* \otimes B(-1) \to \bigwedge^2 E \otimes B,\\
\rS^2 E^* \otimes B' \gets E \otimes E^* \otimes B'(-1) \gets \bigwedge^2 E \otimes B'(-2).
\end{align*}
There is an invariant $K \subset E \otimes E^*$. Since the complexes above are $\GL(E)$-equivariant, we see that $K$ maps to 0 in $\bigwedge^2 E \otimes B$ in the first complex since $B = \rS^\bullet(\rS^2 E^*)$. Similarly, $K$ maps to 0 in $\rS^2 E^* \otimes B'$. So $K[1]$ is a subcomplex of both of the above complexes, and this shows the existence of the $\fpe(V)$-equivariant map $K[1] \to \bigwedge^2 V$. 

The existence of the map $\rS^2 V \to K[1]$ is similar and is obtained by showing that $K[1]$ is a quotient complex of the corresponding 2-sided complexes.
\end{proof}

We can use the trace and evaluation maps to define a larger class of nonzero $\fpe(V)$-equivariant morphisms, which we will need later.

\begin{proposition} \label{prop:pemaps}
Let $\lambda$ be a partition. There are $\fpe(V)$-equivariant morphisms
\[
(\bS_{\lambda / (1,1)} V)[1] \to \bS_\lambda V, \qquad \bS_\lambda V \to (\bS_{\lambda / (2)} V)[1]
\]
which are nonzero in homological degree $0$.
\end{proposition}

\begin{proof}
By \cite[\S 2.4]{weyman}, we can define $\bS_{\alpha / \beta} V$ as a quotient of 
\[
\bigwedge^{\alpha / \beta} V := \bigwedge^{\alpha^\dagger_1 - \beta^\dagger_1} V \otimes \cdots \otimes \bigwedge^{\alpha^\dagger_{\alpha_1} - \beta^\dagger_{\alpha_1}} V.
\]
Consider the diagram
\[
\xymatrix{ 
(\bigwedge^{\lambda / (1,1)} V)[1] \ar[r] \ar[d] & \bigwedge^\lambda V \ar[d] \\
(\bS_{\lambda / (1,1)} V)[1] & \bS_\lambda V
}
\]
where the horizontal map is 
\[
\bigwedge^{\lambda^\dagger_1 - 2} V \otimes K[1] \xrightarrow{{\rm trace}} 
\bigwedge^{\lambda^\dagger_1 - 2} V \otimes \bigwedge^2 V \xrightarrow{m} 
\bigwedge^{\lambda^\dagger_1} V 
\]
($m$ is the multiplication map) tensored with the identity on $\bigwedge^{\lambda^\circ}$ where $\lambda^\circ$ is the diagram of $\lambda$ with the first column removed. We claim that this horizontal map descends to a map which completes the diagram. By \cite[\S 2.4]{weyman}, the kernels of the vertical maps (for general $\alpha / \beta$) are spanned by the subspaces 
\[
\bigwedge^{\alpha^\dagger_1 - \beta^\dagger_1} V \otimes \cdots \otimes R_{a,a+1} V \otimes \cdots \otimes \bigwedge^{\alpha^\dagger_{\alpha_1} - \beta^\dagger_{\alpha_1}} V \qquad (1 \le a \le \alpha_1-1)
\] 
where $R_{a,a+1}V \subset \bigwedge^{\alpha^\dagger_a - \beta^\dagger_a} V \otimes \bigwedge^{\alpha^\dagger_{a+1} - \beta^\dagger_{a+1}} V$ is defined as the span of the images of the maps, for all $u+v < \alpha^\dagger_{a+1} - \beta^\dagger_a$, 
\[
\xymatrix{
\DS \bigwedge^u V \otimes \bigwedge^{\alpha^\dagger_a - \beta^\dagger_a - u + \alpha^\dagger_{a+1} - \beta^\dagger_{a+1} - v} V \otimes \bigwedge^v V \ar[d]^{1 \otimes \Delta \otimes 1} \\
\DS \bigwedge^u V \otimes \bigwedge^{\alpha^\dagger_a - \beta^\dagger_a - u} V \otimes \bigwedge^{\alpha^\dagger_{a+1} - \beta^\dagger_{a+1} - v} V \otimes \bigwedge^v V \ar[d]^{m_{12} \otimes m_{34}} \\
\DS \bigwedge^{\alpha^\dagger_a - \beta^\dagger_a} V \otimes \bigwedge^{\alpha^\dagger_{a+1} - \beta^\dagger_{a+1}} V.}
\]
Here $\Delta$ is comultiplication, and $m_{ij}$ denotes multiplication on the $i$th and $j$th factor. So we just have to check that each such subspace in $\bigwedge^{\lambda/(1,1)} V$ is mapped to another such subspace in $\bigwedge^\lambda V$. These relations only involve two consecutive columns, and $\nu$ and $\lambda$ have the same columns except for the first one, so we only need to check the claim for $a=1$. In this case, it becomes a matter of verifying that the following diagram commutes ($m_{ij}$ is multiplication on the $i$th and $j$th factors)
\[
\footnotesize
\xymatrix{
\DS \bigwedge^u V \otimes \bigwedge^{\lambda^\dagger_1 - 2 - u + \lambda^\dagger_2 - v} V \otimes \bigwedge^v V \ar[d]^{1 \otimes \Delta \otimes 1}[1] \ar[r]^-{\rm trace} &
\DS \bigwedge^u V \otimes \bigwedge^{\lambda^\dagger_1 - 2 - u + \lambda^\dagger_2 - v} V \otimes \bigwedge^2 V \otimes \bigwedge^v V \ar[d]^{1 \otimes \Delta \otimes 1 \otimes 1} \ar[r]^-{m_{23}} &
\DS \bigwedge^u V \otimes \bigwedge^{\lambda^\dagger_1 - u + \lambda^\dagger_2 - v} V \otimes \bigwedge^v V \ar[d]^{1 \otimes \Delta \otimes 1} \\
\DS \bigwedge^u V \otimes \bigwedge^{\lambda^\dagger_1 - 2 - u} V \otimes \bigwedge^{\lambda^\dagger_2 - v} V \otimes \bigwedge^v V \ar[d]^{m_{12} \otimes m_{34}}[1] \ar[r]^-{\rm trace} &
\DS \bigwedge^u V \otimes \bigwedge^{\lambda^\dagger_1 - 2 - u} V \otimes \bigwedge^{\lambda^\dagger_2 - v} V \otimes \bigwedge^2 V \otimes \bigwedge^v V \ar[d]^{m_{12} \otimes m_{35}} \ar[r]^-{m_{24}} & 
\DS \bigwedge^u V \otimes \bigwedge^{\lambda^\dagger_1 - u} V \otimes \bigwedge^{\lambda^\dagger_2 - v} V \otimes \bigwedge^v V \ar[d]^{m_{12} \otimes m_{34}} \\
\DS \bigwedge^{\lambda^\dagger_1 - 2} V \otimes \bigwedge^{\lambda^\dagger_2} V[1] \ar[r]^-{\rm trace} &
\DS \bigwedge^{\lambda^\dagger_1 - 2} V \otimes \bigwedge^2 V \otimes \bigwedge^{\lambda^\dagger_2} V \ar[r]^-{m_{12}} &
\DS \bigwedge^{\lambda^\dagger_1} V \otimes \bigwedge^{\lambda^\dagger_2} V 
}
\]
The commutativity of the left-hand side squares is clear since the maps do not interact in a non-trivial way. The commutativity of the top-right square follows from the compatibility between multiplication and comultiplication in a bialgebra, and the commutativity of the bottom-right square follows from associativity of multiplication.

Dually, recall that we can also define $\bS_\lambda V$ as a
submodule of $\rS^\lambda V := \rS^{\lambda_1} V \otimes \cdots \otimes
\rS^{\lambda_{\ell(\lambda)}} V$. Consider the diagram
\[
\xymatrix{
\bS_\lambda V \ar[d] & \bS_{\lambda / (2)} V[1] \ar[d] \\
\rS^\lambda V \ar[r] & \rS^{\lambda / (2)} V[1]
}
\]
where the bottom horizontal map is induced by the evaluation map $\rS^2 V \to K[1]$. We claim that this descends to a map which completes the diagram. This diagram is dual to one using the Weyl functor presentation for $V^*$ instead of $V$ (see \cite[\S 2.1]{weyman}; the definition of Weyl functor can easily be extended to ``Weyl complex''). The relations defining Weyl functors are analogous to the ones defining the Schur functors. So the proof that the map descends reduces to the commutativity of a similar 9-term diagram. We only need that multiplication and comultiplication make the divided power algebra into a bialgebra.
\end{proof}

\begin{example}
When $\lambda = (2,1)$, the composition of trace and evaluation is an isomorphism. If $\{e_1, \dots, e_n\}$ is a basis for $E$, then the trace map is
\[
e_j \mapsto \sum_{i=1}^n (e_i \otimes e_i^*) \otimes e_j + \sum_{i=1}^n (e_i^* \otimes e_i) \otimes e_j.
\]
The evaluation map pairs the first and third tensor factors, so the first sum becomes 0 and the second sum becomes $e_j$. Similarly, the composition maps all dual basis vectors to themselves. As a consequence, $V[1]$ is a $\fpe$-equivariant direct summand of $\bS_{2,1} V$.
\end{example}

We make the following definitions:
\begin{equation} \label{eqn:kidefn}
\begin{split}
k_\lambda(V) &= \ker(\bS_\lambda V \to (\bS_{\lambda / (2)} V)[1]),\\ 
i_\lambda(V) &= \im((\bS_{\lambda / (1,1)} V)[1] \to \bS_\lambda V),\\ 
\bS_{[\lambda]} V &= k_\lambda(V) / (k_\lambda(V) \cap i_\lambda(V)).
\end{split}
\end{equation}

\subsection{Existence of representation structure.} \label{sec:peri_exist}

Recall the definition of $P_{r,s}(\alpha)$ from \eqref{eqn:Prs}. Consider the partition $P_{r,s}(\emptyset) = (s^{s+r})$. Then we have $\bS_\lambda E = \bS_{(s^{s+r})} E^* \otimes (\det E)^s$ where 
\[
\lambda = (s^{\dim E -s-r})
\]
\cite[Exercise 2.18]{weyman}. We will repeatedly use the fact that if $\mu = (a^b)$ is a rectangular shape, and $\nu \subseteq \mu$, then $\bS_{\mu / \nu} = \bS_\eta$ where $\eta = (a-\nu_b, a-\nu_{b-1}, \dots, a-\nu_1)$. This follows by showing that they have the same character using semistandard Young tableaux \cite[Proposition 2.1.15]{weyman}.

We will also make use of the following simple consequence of the Littlewood--Richardson rule \cite[Theorem 2.3.4]{weyman}: if $\eta^1, \dots, \eta^d$ are partitions, then $\bS_{\eta^1 + \cdots + \eta^d} U$ appears with multiplicity 1 in $\bS_{\eta^1} U \otimes \cdots \otimes \bS_{\eta^d} U$. We define this to be the {\bf Cartan product}.

Recall that $B = \rS^\bullet(\rS^2 E^*)$. We will use the 2-sided complex interpretation of $\bS_\lambda V$ and $\bS_{[\lambda]} V$ to treat them as complexes over $B$. We define 
\begin{align*}
\tilde{\bS}_{\lambda} V &= \bS_{\lambda} V \otimes (\det E^*)^s,\\
\tilde{\bS}_{[\lambda]} V &= \bS_{[\lambda]} V \otimes (\det E^*)^s,
\end{align*} 
so that $(\tilde{\bS}_{[\lambda]} V)_0 = (\tilde{\bS}_\lambda V)_0 = \bS_{(s^{s+r})} E^* \otimes B$. We use $(\Phi, \Phi')$ to denote the 2-sided complex structure on $\tilde{\bS}_{\lambda} V$ and $\tilde{\bS}_{[\lambda]} V$. 

\begin{lemma} \label{lem:symmetricshapes}
If $s+r+\alpha_1 \le \dim E$, then $\bS_{P_{r,s}(\alpha)} E^*$ appears with  multiplicity $1$ in $(\tilde{\bS}_\lambda V)_{|\alpha|}$ and also in $(\tilde{\bS}_{[\lambda]} V)_{|\alpha|}$.
\end{lemma}

\begin{proof}
The inequality just means that $\bS_{P_{r,s}(\alpha)} E^* \ne 0$.

The $i$th term of the Schur complex $\bS_\mu(W_0 \oplus W_1)$ is $\bigoplus_{\nu \subseteq \mu,\ |\nu|=i} \bS_{\mu/\nu} W_0 \otimes \bS_{\nu^\dagger} W_1$ \cite[Theorem 2.4.5]{weyman}. When $\mu = \lambda = (s^{\dim E - s - r})$ and $W_0 = E$, we see from the above discussion that $\bS_{\lambda/\nu} E \cong \bS_{(s^{s+r}, \nu)} E^*$. 

A simple consequence of the Littlewood--Richardson rule \cite[Theorem 2.3.4]{weyman} is that if $\bS_\eta U$ appears in $\bS_{\eta'} U \otimes \bS_{\eta''} U$, then $\eta' \subseteq \eta$ and $\eta'' \subseteq \eta$. In particular, if we choose $\nu$ of size $|\alpha|$ as above, we can only get $\bS_{P_{r,s}(\alpha)} E^* \subset \bS_{\lambda/\nu} E \otimes \bS_{\nu^\dagger} E^*$ if $\nu = \alpha^\dagger$. Taking the Cartan product of $\bS_{\lambda / \alpha^\dagger} E = \bS_{(s^{r+s}, \alpha^\dagger)} E^*$ and $\bS_{\alpha} E^*$ gives that $\bS_{P_{r,s}(\alpha)} E^*$ appears in $(\tilde{\bS}_\lambda V)_{|\alpha|}$ with multiplicity 1.

We also see that $\bS_{P_{r,s}(\alpha)} E^*$ does not appear in either $\bS_{\lambda / (2)} V$ nor $\bS_{\lambda / (1^2)} V$. This shows that $\bS_{P_{r,s}(\alpha)} E^*$ also appears in $(\tilde{\bS}_{[\lambda]} V)_{|\alpha|}$ with multiplicity 1.
\end{proof}

\begin{lemma} \label{lem:surj}
Pick $\alpha$ with $\ell(\alpha) \le s$, and pick $\beta \subset \alpha$ such that $|\alpha|-1=|\beta|$. Then $\bS_{P_{r,s}(\beta)} E^*$ is in the image of $\Phi \colon \bS_{P_{r,s}(\alpha)} E^* \otimes \rS^2 E \to (\tilde{\bS}_\lambda V)_{|\beta|}$. Similarly, $\bS_{P_{r,s}(\alpha)} E^*$ is in the image of $\Phi' \colon \bS_{P_{r,s}(\beta)} E^* \otimes \bigwedge^2 E^* \to (\tilde{\bS}_\lambda V)_{|\alpha|}$.
\end{lemma}

\begin{proof}
We only prove the statement about $\Phi$ since the proof for $\Phi'$ is similar.

Set $W_r = \bigwedge^r V \otimes \det E^*$. For the case $s=1$, the Schur complex $\tilde{\bS}_\lambda V$ is $W_r$. It was analyzed in Example~\ref{eg:s=1}, and the lemma holds by inspection in this case.

For the general case, consider the quotient map $\pi \colon W_r^{\otimes s} \to \tilde{\bS}_\lambda V$ \cite[\S 2.4]{weyman}. The first $s$ column lengths of $P_{r,s}(\alpha)$ are $c_1 := r + s + \alpha_1, \dots, c_s := r + s + \alpha_s$. By \cite[Exercise 2.18]{weyman}, we have
\[
\bigwedge^{\dim E - c_i} E \otimes \rS^{\alpha_i} E^* = \bigwedge^{c_i} E^* \otimes \rS^{\alpha_i} E^*,
\]
which is naturally a subspace of $(W_r)_{\alpha_i}$ \cite[Theorem 2.4.5]{weyman}. 

We claim that the product of the $\bigwedge^{c_i} E^* \otimes \rS^{\alpha_i} E^*$ generates $\bS_{P_{r,s}(\alpha)} E^*$ in $\tilde{\bS}_\lambda V$. To see this, first replace $V_1 = E^*$ with an arbitrary vector space $F$. Then repeating the above, we are considering the product of $\bigwedge^{c_i} E^* \otimes \rS^{\alpha_i} F$. The Cartan product of the $\bigwedge^{c_i} E^*$ is $\bS_\mu E^*$ where $\mu$ consists of the first $s$ columns of $P_{r,s}(\alpha)$. Also, $\bS_\mu E^* \otimes \bS_\alpha F$ is the unique summand in $\tilde{\bS}_\lambda V$ which contains a $\bS_\mu E^*$-isotypic component, so this must be in the image of $\pi$. Finally, note that $\bS_{P_{r,s}(\alpha)} E^*$ is the Cartan product of $\bS_\mu E^*$ and $\bS_\alpha E^*$. This proves the claim.

To go from $\alpha$ to $\beta$, we decrease $c_i$ by 1, where $i$ is the unique column index in which $\alpha$ and $\beta$ differ. Consider the map on $W_r^{\otimes s}$ which is $\Phi$ on the $i$th factor and the identity elsewhere. Descending this map to $\tilde{\bS}_\lambda V$, we see that $\bS_{P_{r,s}(\beta)} E^*$ is in the image of $\Phi$ restricted to $\bS_{P_{r,s}(\alpha)} E^* \otimes \rS^2 E^*$.
\end{proof}

\begin{proof}[Proof of Theorem~\ref{thm:periplecticaction}]
In \S\ref{sec:per_jpw}, we discussed that the sequences $\JPW^{r,s}_\bullet \otimes B$ can be given the structure of linearly exact $B$-complexes. Let $\uJPW^{r,s}_\bullet$ denote this $B$-complex. We will construct a map of complexes $\tilde{\bS}_{[\lambda]} V \to \uJPW^{r,s}_\bullet$. The presentation for $\rH_0(\uJPW^{r,s}_\bullet)$ is
\[
\bS_{(s+1, s^{s-1+r}, 1)} E^* \otimes B(-1) \to \bS_{(s^{s+r})} E^* \otimes B.
\]
Recall the definitions from \eqref{eqn:kidefn}. By Proposition~\ref{prop:pemaps}, we have
\begin{align*}
(\tilde{\bS}_\lambda V)_1 &= \bS_{(s+1,s^{s-1+r},1)} E^* \oplus \bS_{(s^{s+r},2)} E^* \oplus \bS_{(s^{s+r}, 1,1)} E^*,\\
(k_\lambda V \otimes (\det E^*)^s)_1 &= \bS_{(s+1,s^{s-1+r},1)} E^* \oplus \bS_{(s^{s+r}, 1,1)} E^*,\\
(i_\lambda V \otimes (\det E^*)^s)_1 &= \bS_{(s+1,s^{s-1+r},1)} E^* \oplus \bS_{(s^{s+r},2)} E^*,\\
(\tilde{\bS}_{[\lambda]} V)_1 &= \bS_{(s+1,s^{s-1+r},1)} E^*.
\end{align*}

The possible maps $(\tilde{\bS}_{[\lambda]} V)_1 \to (\tilde{\bS}_{[\lambda]} V)_0$ are unique up to a choice of scalar, so it only matters if it is zero or nonzero. In either case, we have a natural surjection $f_{-1} \colon \rH_0(\tilde{\bS}_{[\lambda]} V) \to \rH_0(\uJPW^{r,s}_\bullet)$. For notation, set $(\tilde{\bS}_{[\lambda]} V)_{-1} = \rH_0(\tilde{\bS}_{[\lambda]} V)$ and $\uJPW^{r,s}_{-1} = \rH_0(\uJPW^{r,s}_\bullet)$. We will construct maps $f_i \colon \tilde{\bS}_{[\lambda]} V \to \uJPW^{r,s}_{-1}$ by induction on $i$ satisfying 
\begin{compactitem}
\item $f_i$ is $\fgl(E)$-equivariant,
\item $f_i$ is surjective,
\item the $f_j$ for $j \le i$ form a morphism $f_{\le i}$ of the truncated complexes,
\item for all $x \in \ker f_{i-1}$ and $Y \in \fpe(V)_1$, we have $\Phi'(Y)(x) \in \ker f_i$
\end{compactitem}

These conditions hold for $i=-1$, which handles the base case of our induction. Assuming that we have constructed $f_i$, we show how to construct $f_{i+1}$. First pick $x \in \ker f_i$ and $Y \in \fpe(V)_1$. Since $\ker f_i$ is $\fgl(E)$-equivariant and
\[
[\Phi(X), \Phi'(Y)] = \Phi(X) \Phi'(Y) + \Phi'(Y) \Phi(X) \in \fpe(V)_0 = \fgl(E)
\]
for all $X \in \fpe(V)_{-1}$, we have
\[
(\Phi(X) \Phi'(Y) + \Phi'(Y) \Phi(X))(x) \in \ker f_i.
\]
Since $f_{\le i}$ is a morphism of complexes, $\Phi(X)(x) \in \ker f_{i-1}$, which implies that $\Phi'(Y) \Phi(X)(x) \in \ker f_i$ by our conditions on $f_{\le i}$. In particular, we also have $\Phi(X)\Phi'(Y)(x) \in \ker f_i$. Hence the subspace
\[
U_{i+1} = \{\Phi'(Y)(x) \mid Y \in \fpe(V)_1, \ x \in \ker f_i\}
\]
is sent to 0 under the composition
\[
(\tilde{\bS}_{[\lambda]} V)_{i+1} \xrightarrow{\Phi} (\tilde{\bS}_{[\lambda]} V)_i  \xrightarrow{f_i} \uJPW_i^{r,s}.
\]
Again since $f_{\le i}$ is a morphism of complexes, the image of this map is in the kernel of the differential $\uJPW_i^{r,s} \to \uJPW_{i-1}^{r,s}$. More specifically, the generators map to the degree 1 piece of this kernel. Since $\uJPW_\bullet^{r,s}$ is linearly exact, we can find a lift 
\[
(\tilde{\bS}_{[\lambda]} V)_{i+1} / U_{i+1} \to \uJPW_{i+1}^{r,s}.
\]
Since everything above is $\fgl(E)$-equivariant, we can choose this lift to also be $\fgl(E)$-equivariant by invoking the semisimplicity of $\fgl(E)$. We define $f_{i+1} \colon (\tilde{\bS}_{[\lambda]} V)_{i+1} \to \uJPW_{i+1}^{r,s}$ by composing with the quotient map. Using Lemma~\ref{lem:symmetricshapes}, Lemma~\ref{lem:surj}, and the explicit definition \eqref{eqn:jpwdefn}, we deduce that $f_{i+1}$ is surjective from the fact that $f_i$ is surjective. The other conditions on $f_{i+1}$ follow by construction.

So we have constructed a surjective $B$-degree 0 map of complexes
\begin{align} \label{eqn:fmap}
f \colon \tilde{\bS}_{[\lambda]} V \to \uJPW^{r,s}_\bullet.
\end{align}
Since $\ker f$ is a $\fpe(V)$-subrepresentation of $\tilde{\bS}_{[\lambda]} V$,  we conclude that $\uJPW_\bullet^{r,s}$ has the structure of a 2-sided complex, and hence that $\JPW_\bullet^{r,s}$ is a representation of $\fpe(V)$. All that remains is to show that this action makes $\JPW_\bullet^{r,s}$ an irreducible representation. So let $\bF_\bullet$ be any nonzero submodule of $\uJPW_\bullet^{r,s}$ and consider its preimage under $f$. Using Lemma~\ref{lem:surj}, we conclude that this preimage contains all $\bS_{P_{r,s}(\alpha)} E^*$, so the same is true for $\bF_\bullet$, and hence $\bF_\bullet = \uJPW^{r,s}_\bullet$.
\end{proof}

\begin{conjecture}
The map $f$ from \eqref{eqn:fmap} is an isomorphism.
\end{conjecture}

\begin{remark}
For $r$ odd, we have linear maps
\[
\JPW^{r,s}_i \otimes \bigwedge^2 E \to \JPW^{r,s}_{i-1}
\]
which come from the minimal free resolutions of Pfaffian ideals \cite[\S 6.4]{weyman}. If we had defined the periplectic superalgebra as the stabilizer of an odd symmetric bilinear form rather than a skew-symmetric one in \S\ref{sec:periplectic}, we would get an isomorphic algebra with the $\bZ$-grading reversed (and the roles of $E$ and $E^*$ swapped). We could have used this other direction to define an action of $\fpe(V)$ on $\JPW_\bullet^{r,s}$.
\end{remark}

\section{Generic minors.} \label{sec:genericminors}

Recall that we defined $A = \rS^\bullet(\hom(F,E)^*)$ and $A' = \rS^\bullet(\hom(E,F)^*)$ in \S\ref{sec:gl_2sided}.

\subsection{Lascoux complexes.} \label{sec:lascoux_complex}

Choose integers $s \ge 0$ and $r > 0$. Given partition $\alpha, \beta$ with $\ell(\alpha) \le s$ and $\beta_1 \le s$, define the partitions
\begin{equation} \label{eqn:PQdefn}
\begin{split}
P_{r,s}(\alpha, \beta) &= (s + \alpha_1, \dots, s + \alpha_s, s^r, \beta_1, \dots, \beta_{\ell(\beta)}),\\
Q_{r,s}(\alpha, \beta) &= (s + \beta^\dagger_1, \dots, s + \beta^\dagger_s, s^r, \alpha^\dagger_1, \dots, \alpha^\dagger_{\alpha_1}),\\
\end{split}
\end{equation}
which we visualize as follows:
\[
\begin{tikzpicture}[scale=.4]
\draw (0,0) -- (0,7) -- (8,7) -- (8,6) -- (6,6) -- (6,5) -- (5,5) -- (5,4) -- (3,4) -- (3,2) -- (2,2) -- (2,0) -- (0,0) -- cycle;
\draw (0,3) -- (3,3);
\draw (0,4) -- (3,4) -- (3,7);
\path (1.5,5.5) node {$s \times s$};
\path (1.5,3.5) node {$r \times s$};
\path (4.5,5.5) node {$\alpha$};
\path (1,1.5) node {$\beta$};
\draw (12,-2) -- (12,7) -- (18,7) -- (18,5) -- (16,5) -- (16,4) -- (15,4) -- (15,1) -- (14,1) -- (14,0) -- (13,0) -- (13,-2) -- cycle;
\draw (12,4) -- (15,4) -- (15,7);
\draw (12,3) -- (15,3);
\path (13.5,5.5) node {$s \times s$};
\path (13.5,3.5) node {$r \times s$};
\path (16.5,6) node {$\beta^\dagger$};
\path (13.5,1.5) node {$\alpha^\dagger$};
\end{tikzpicture}
\]

Set 
\begin{align}
\La^{r,s}_i = \bigoplus_{|\alpha| + |\beta| = i} \bS_{P_{r,s}(\alpha, \beta)} E^* \otimes \bS_{Q_{r,s}(\alpha, \beta)} F
\end{align}
which naturally carries an action of $\fgl(E) \times \fgl(F)$. Up to a homological grading shift, the sequences $\La^{r,s}_\bullet \otimes A$ can be realized as the linear strands of the ideal of $(r+1) \times (r+1)$ minors of the generic matrix $\hom(F,E)$ \cite[\S 6.1]{weyman} (we disallowed the case $r=0$ because it corresponds to the Koszul complex, which is a degenerate case). This definition of this complex was given by Lascoux \cite{lascoux}.

As a consequence, we get $A$-linear maps
\[
\Phi_i \colon \La^{r,s}_i \otimes A(-1) \to \La^{r,s}_{i-1} \otimes A.
\]
The main result in this section is that $\Phi$ can be completed to a 2-sided complex $(\Phi, \Phi')$.

\begin{theorem} \label{thm:lascoux_action}
There exist $A'$-linear maps
\begin{align*}
\Phi'_i \colon \La_i^{r,s} \otimes A'(-1) \to \La_{i+1}^{r,s} \otimes A'
\end{align*}
so that $(\Phi, \Phi')$ is a $2$-sided $\fgl(V)$-complex. In particular, $\La_\bullet^{r,s}$ affords an action of $\fgl(E)$. Moreover, $\La_\bullet^{r,s}$ is an irreducible $\fgl(V)$-representation.
\end{theorem}

The proof will be given in \S\ref{sec:gl_exist}.

\subsection{Trace and evaluation.}

Let $e_1, \dots, e_n$ be a basis for $E$ and let $f_1, \dots, f_m$ be a basis for $F$. The degree 0 piece of $V \otimes V^*$ is $(E \otimes E^*) \oplus (F \otimes F^*)$, and we define the {\bf trace} map
\begin{align*}
K &\to (E \otimes E^*) \oplus (F \otimes F^*)\\
\alpha &\mapsto \alpha\sum_i e_i \otimes e_i^* - \alpha \sum_j f_j \otimes f_j^*.
\end{align*}

We also define the {\bf evaluation} map
\begin{align*}
(E \otimes E^*) \oplus (F \otimes F^*) &\to K\\
\sum_{i,j} a_{i,j} e_i \otimes e_j^* + \sum_{k, \ell} b_{k,\ell} f_k \otimes f_\ell^* &\mapsto a_{1,1} + \cdots + a_{n,n} + b_{1,1} + \cdots + b_{m,m}.
\end{align*}

These maps also appeared in \cite{pw}.

\begin{proposition}
Trace and evaluation give nonzero $\fgl(V)$-equivariant maps $K \to V \otimes V^*$ and $V \otimes V^* \to K$, where $K$ denotes the trivial $1$-dimensional module.
\end{proposition}

\begin{proof}
It is straightforward to check that trace is $\fgl(E) \times \fgl(F)$-equivariant and that the image is in the kernel of the map 
\begin{align*}
(E \otimes E^*) \oplus (F \otimes F^*) \xrightarrow{1 \otimes \Phi(X) + \Phi(X)^* \otimes 1 } (E \otimes F^*)
\end{align*}
for any $X \in \hom(F,E)$, and is in the kernel of the map 
\[
(E \otimes E^*) \oplus (F \otimes F^*) \xrightarrow{\Phi'(Y) \otimes 1 + 1 \otimes \Phi'(Y)^*} F \otimes E^*
\]
for any $Y \in \hom(E,F)$. Using the 2-sided complex interpretation, this shows that the image of $K$ is a $\fgl(V)$-submodule of $V \otimes V^*$. 

The analysis of the evaluation map is similar.
\end{proof}

Given two partitions $\lambda$ and $\mu$, we can define maps
\begin{align} \label{eqn:glmorphisms}
(\bS_{\lambda/1} V \otimes \bS_{\mu/1}(V^*[1]))[1] \to \bS_\lambda V \otimes \bS_\mu (V^*[1]) \to (\bS_{\lambda/1} V \otimes \bS_{\mu/1} (V^*[1]))[1]
\end{align}
which are induced by trace and evaluation. The construction is analogous to  the one in \S\ref{sec:periplectic:morphisms}.

\begin{proposition} \label{prop:glcomplex}
If $\dim E - \lambda^\dagger_1 + \lambda_1 = \dim F - \mu^\dagger_1 + \mu_1$, then the composition \eqref{eqn:glmorphisms} is $0$ in homological degree $1$.
\end{proposition}

\begin{example}
When $\lambda = \mu = (1)$, this is saying that the composition of trace and evaluation is 0 when $\dim E = \dim F$, which is easily seen.
\end{example}

\begin{proof} 
Using the standard basis of \cite[Proposition 2.1.15(b)]{weyman}, an element in homological degree 1 in the left hand side of \eqref{eqn:glmorphisms} is a sum of pairs of tableaux $(T_e, T_f)$ of shapes $(\lambda/1, \mu/1)$ and whose entries are filled with $\{e_1, \dots, e_n\}$ and $\{f^*_1, \dots, f^*_m\}$, respectively. So fix such a pair. Now we apply the map \eqref{eqn:glmorphisms}. 

The trace map says to sum over the tableaux we get by inserting $e_i \otimes e_i^*$ and $-f_j \otimes f_j^*$ into the empty boxes of $(T_e, T_f)$. When we insert $v$ into the empty box of $T_e$, we denote it by $vT_e$ (same for $T_f$). So we can write the trace map as 
\[
(T_e, T_f) \mapsto \sum_i (e_iT_e, e_i^*T_f) - \sum_j (f_j T_e, f_j^* T_f)
\]
The next part of the map tells us to antisymmetrize all columns. In detail, for each pair of permutations $(\sigma, \rho)$ of the boxes of $\lambda$ and $\mu$, we sum over those which preserve the columns, and multiply by the sign of the permutation. In symbols:
\[
(e_i T_e, e_i^* T_f) \mapsto \sum_{\sigma, \rho} {\rm sgn}(\sigma) {\rm sgn}(\rho) (\sigma \cdot e_i T_e, \rho \cdot e_i^* T_f).
\]
Fix a term in this sum, we will show that its image is 0.

To each term in this sum, we symmetrize rows (i.e., interpret them as monomials in symmetric powers) and then pick one element from the first row of both tableaux in all possible ways and evaluate them against one another. When we do this to $(e_iT_e, e_i^*T_f)$, we can only get a nonzero result if we pick $e_i^*$ in the first row of $e_i^* T_f$ (this might not be possible depending on the signed term we picked from the antisymmetrization). When we pick this $e_i^*$, then the sum of all possible evaluations is 1 plus the number of times that $e_i$ appears in the first row of the fixed antisymmetrization of $T_e$. The only $i$ that can contribute are those such that $e_i$ does not appear in the first column of $T_e$ (otherwise $e_iT_e$ would have two instances of $e_i$ in the same column and be identically 0). So summing over all $i$ such that $i$ is not in the first column of $T_e$, we get $(\dim E - \lambda^\dagger_1 + 1) + (\lambda_1 - 1) = \dim E - \lambda^\dagger_1 + \lambda_1$ contributions. Similarly, considering $(f_j T_e, f_j^* T_f)$, we get $\dim F - \mu^\dagger_1 + \mu_1$ contributions, each having coefficient $-1$. So the total coefficient is 0.
\end{proof}

We make the following definitions:
\begin{equation}
\begin{split}
k_{\lambda; \mu}(V) &= \ker(\bS_\lambda V \otimes \bS_\mu (V^*[1]) \to (\bS_{\lambda/1} V \otimes \bS_{\mu/1} (V^*[1]))[1])\\
i_{\lambda;\mu}(V) &= \im((\bS_{\lambda/1} V \otimes \bS_{\mu/1} (V^*[1]))[1] \to \bS_\lambda V \otimes \bS_\mu (V^*[1]))\\
\bS_{[\lambda; \mu]} V &= k_{\lambda; \mu}(V) / (k_{\lambda; \mu}(V) \cap i_{\lambda; \mu}(V)).
\end{split}
\end{equation}

\subsection{Existence of representation structure.} \label{sec:gl_exist}

\begin{proof}[Proof of Theorem~\ref{thm:lascoux_action}] 
Recall the definitions of $P_{r,s}(\alpha,\beta)$ and $Q_{r,s}(\alpha, \beta)$ from \eqref{eqn:PQdefn}. Take $\lambda = (s^{\dim E-r-s})$ and $\mu = (s^{\dim F-r-s})$. Then 
\begin{align*}
\bS_\lambda E &= \bS_{P_{r,s}(\emptyset, \emptyset)} E^* \otimes (\det E)^s\\
\bS_\mu F^* &= \bS_{Q_{r,s}(\emptyset, \emptyset)} F \otimes (\det F^*)^s
\end{align*}
\cite[Exercise 2.18]{weyman}. In \S\ref{sec:lascoux_complex}, we discussed that $\La_\bullet^{r,s} \otimes A$ can be realized as a linearly exact $A$-linear complex, which we denote by $\uLa_\bullet^{r,s}$. We now interpret $\bS_\lambda V$, $\bS_\mu (V^*[1])$, and $\bS_{[\lambda;\mu]} V$ as 2-sided complexes. The first two terms of $\bS_\lambda V \otimes (\det E^*)^s$ are (write $(\alpha; \beta)$ in place of $\bS_\alpha E^* \otimes \bS_\beta F$)
\[
((s^{s+r},1);1) \to (s^{s+r}; \emptyset)
\]
Similarly, the first two terms of $\bS_\mu (V^*[1]) \otimes (\det F)^s$ are
\[
(1; (s^{s+r},1)) \to (\emptyset; s^{s+r})
\]
So the first two terms of $\bS_{\lambda} V \otimes \bS_{\mu} (V^*[1]) \otimes (\det E^*)^s \otimes (\det F)^s$ are 
\[
\begin{array}{c} ((s^{s+r},1); (s+1,s^{s+r-1})) \oplus \\ 
((s^{s+r},1); (s^{s+r},1)) \oplus \\ 
((s^{s+r},1); (s^{s+r},1)) \oplus \\ 
((s+1,s^{s+r-1}); (s^{s+r},1))
\end{array} \to (s^{s+r}; s^{s+r})
\]
Also, $((s^{s+r},1); (s^{s+r},1))$ is the 0th term of $\bS_{\lambda/1} V \otimes \bS_{\mu/1}(V^*[1]) \otimes (\det E^*)^s \otimes (\det F)^s$. Our choice of $\lambda$ and $\mu$ satisfies Proposition~\ref{prop:glcomplex}, so neither instance of $((s^{s+r},1); (s^{s+r},1))$ remains when we pass to $\tilde{\bS}_{[\lambda; \mu]} V := \bS_{[\lambda; \mu]} V \otimes (\det E^*)^s \otimes (\det F)^s$.

Hence we get a surjection 
\[
\rH_0(\tilde{\bS}_{[\lambda; \mu]} V) \to \rH_0(\uLa_\bullet^{r,s}).
\]
The rest of the proof is essentially the same as the proof of Theorem~\ref{thm:periplecticaction}, so we omit the details.
\end{proof}

\bigskip

\small \noindent Steven V Sam, Department of Mathematics,
University of California, Berkeley, CA \\
{\tt svs@math.berkeley.edu}, \url{http://math.berkeley.edu/~svs/}

\end{document}